\documentclass[11pt]{amsart}
\usepackage[margin=1.4in]{geometry}

\usepackage[T1]{fontenc}
\usepackage{lmodern,microtype}
\usepackage{amsmath,amssymb,amsthm,mathtools}
\usepackage{xcolor}
\usepackage[colorlinks=true,linkcolor=blue!50!black,
  citecolor=green!35!black,urlcolor=blue!50!black]{hyperref}
\theoremstyle{plain}
\newtheorem{theorem}{Theorem}[section]
\newtheorem{proposition}[theorem]{Proposition}
\newtheorem{lemma}[theorem]{Lemma}
\newtheorem{corollary}[theorem]{Corollary}
\theoremstyle{definition}
\newtheorem{example}[theorem]{Example}
\newtheorem{remark}[theorem]{Remark}

\newcommand{\C}{\mathbb C}
\newcommand{\D}{\mathbb D}
\newcommand{\Nzero}{\mathbb N_0}
\newcommand{\Hh}{\mathcal H}
\newcommand{\M}{\mathcal M}
\newcommand{\A}{\mathcal A}
\newcommand{\Span}{\operatorname{span}}
\newcommand{\codim}{\operatorname{codim}}
\newcommand{\rank}{\operatorname{rank}}
\newcommand{\norm}[1]{\left\lVert#1\right\rVert}
\newcommand{\abs}[1]{\left\lvert#1\right\rvert}
\newcommand{\inner}[2]{\langle#1,#2\rangle}
\newcommand{\jury}{\diamond}

\title[Finite-orbit obstructions on the polydisk]
{Finite-Orbit Obstructions for Multipliers on the Polydisk}
\author[I.~Krishtal]{Ilya Krishtal}
\address{School of Mathematical and Statistical Sciences,
Northern Illinois University, DeKalb, IL 60115, USA}
\email{ikrishtal@niu.edu}
\urladdr{https://orcid.org/0000-0001-7171-2177}
\author[J.~Mashreghi]{Javad Mashreghi}
\address{D\'epartement de math\'ematiques et de statistique,
Universit\'e Laval, Qu\'ebec, QC G1V 0K6, Canada}
\email{javad.mashreghi@ulaval.ca}
\urladdr{https://orcid.org/0000-0002-7969-9576}

\author[P.~K.~Vishwakarma]{Prateek Kumar Vishwakarma}
\address{D\'epartement de math\'ematiques et de statistique,
Universit\'e Laval, Qu\'ebec, QC G1V 0K6, Canada}
\email{prateek-kumar.vishwakarma.1@ulaval.ca}
\urladdr{https://orcid.org/0000-0001-7502-1381}
\thanks{Corresponding author: Ilya Krishtal
(\href{mailto:ikrishtal@niu.edu}{ikrishtal@niu.edu}).}
\date{September 2026}
\hypersetup{pdftitle={Finite-Orbit Obstructions for Multipliers on the Polydisk},
pdfauthor={Ilya Krishtal, Javad Mashreghi, Prateek Kumar Vishwakarma}}

\begin{document}
\raggedbottom
\begin{abstract}
We study the closed spans of finitely many joint orbits of holomorphic
multipliers on the Hardy space $H^2(\D^d)$ of the polydisk. If fewer
than $d$ symbols are used, every such span has infinite codimension.
The symbols may be unbounded, provided that all orbit vectors belong
to the Hardy space. We give two elementary proofs. The first uses
common level sets and independent point evaluations; for bounded
symbols, these evaluations yield joint adjoint eigenvectors. The
second uses finite Taylor sections and the nilpotent structure of
truncated convolution, also known as the Jury product. For $r$
generators and $q<d$ symbols, the orbit dimension on a cube of side
$n$ is $O(n^q)$, whereas the ambient dimension is $n^d$. In the bidisk
we obtain the explicit codimension bound $MN-r(M+N-1)$, which is
sharp for a single orbit. The coordinate multipliers attain the
parameter threshold. We also explain how these obstructions relate
to kernel interpolation and the established representation of
dynamical frames by compressed shifts on model spaces.
\end{abstract}
\maketitle

\section{Introduction}

A familiar question in dynamical sampling asks whether repeated
application of an operator can compensate for having only a few
initial vectors. Given a bounded operator $T$ on a Hilbert space
$\Hh$, one studies families
\begin{equation}\label{eq:orbit}
   \{T^n f_j:n\in\Nzero,\ 1\leq j\leq r\}.
\end{equation}
Such a family is a frame if there are constants $0<A\leq B<\infty$
such that
\[
 A\norm{h}^2\leq
 \sum_{j=1}^r\sum_{n\geq0}\abs{\inner{h}{T^nf_j}}^2
 \leq B\norm{h}^2\qquad(h\in\Hh).
\]
For background on frames and Riesz sequences, see \cite{Christensen}.
Completeness is the weaker requirement that the closed linear span
be all of $\Hh$. An operator admitting a complete family
\eqref{eq:orbit} with finitely many initial vectors is called
\emph{finitely multicyclic}; see \cite{Trivedi} for this notion and
its extension to commuting tuples.

On the one-variable Hardy space, multiplication by $z$ supplies the
simplest example: the orbit of the constant function $1$ is the
monomial orthonormal basis. The situation changes for scalar
multipliers on a polydisk. Aguilera and Carando
\cite[Proposition~2.5]{AguileraCarando} proved that if $d\geq2$ and
$\varphi\in H^\infty(\D^d)$, no finite family of orbits of
$M_\varphi f=\varphi f$ is a frame for $H^2(\D^d)$. We show that
the obstruction occurs already at the level of completeness:
\begin{equation}\label{eq:scalar-result}
 \codim\overline{\Span}
 \{\varphi^nf_j:n\in\Nzero,\ 1\leq j\leq r\}=\infty
 \qquad(d\geq2).
\end{equation}
There is no assumption on the norms of the iterates or on possible
frame bounds. In the bounded scalar case, \eqref{eq:scalar-result}
also follows from the classical Berger--Shaw theorem
\cite{BergerShaw}; we record this consequence in
Remark~\ref{rem:berger-shaw}. Our emphasis is on elementary proofs,
joint holomorphic symbols, and explicit estimates for finite
Taylor sections.

The two arguments describe the obstruction in complementary ways.
The first takes place in the domain of the functions. A scalar
holomorphic function of several variables has a level set containing
infinitely many points. Evaluation at those points gives many
eigenvectors for the adjoint multiplier, and finitely many initial
vectors cannot account for all of them. The second argument takes
place among Taylor coefficients. On a finite rectangle,
multiplication becomes truncated convolution. After subtracting its
constant term, the multiplier becomes nilpotent. Consequently, one
orbit explores only a small part of a sufficiently large rectangle.

The latter viewpoint connects the problem with the Jury product.
Jury introduced a matrix-product framework in his work on
interpolation problems \cite{JuryThesis}. The recent study of
matrix convolution by Mashreghi, Nasri, and Vishwakarma
\cite{MashreghiNasriVishwakarma} includes the Cayley--Hamilton identity
that governs our finite sections. We explain this connection in
Section~\ref{sec:sections}, including the exact number of initial
matrices needed to generate a finite rectangle. For $r$ scalar
orbits, the resulting codimension estimate on an $n\times n$ section
is $n^2-r(2n-1)$; equality occurs for one orbit of $z_1+z_2$.

Both arguments apply to joint orbits of $q<d$ holomorphic symbols,
including unbounded ones when all orbit vectors belong to $H^2$;
see Corollary~\ref{cor:unbounded}. On cubic Taylor sections of side
$n$, the orbit dimension is $O(n^q)$ inside an ambient space of
dimension $n^d$. The $d$ coordinate multipliers attain the parameter
threshold. Section~\ref{sec:kernels} recalls kernel interpolation
and compressed-shift models of dynamical frames, placing the
obstruction in the context of optimal generators and sectorial
temporal sampling.

\section{Level sets and adjoint eigenvectors}\label{sec:fibers}

We use the coefficient description of the Hardy space
\cite{RudinPolydiscs}:
\[
 H^2(\D^d)=\left\{f(z)=\sum_{\alpha\in\Nzero^d}a_\alpha z^\alpha:
      \norm{f}^2=\sum_\alpha\abs{a_\alpha}^2<\infty\right\}.
\]
Here $z^\alpha=z_1^{\alpha_1}\cdots z_d^{\alpha_d}$ and
$\abs{\alpha}=\alpha_1+\cdots+\alpha_d$. Its reproducing kernel at
$w\in\D^d$ is
\begin{equation}\label{eq:kernel}
 K_w(z)=\prod_{\ell=1}^d\frac{1}{1-z_\ell\overline{w_\ell}},
 \qquad
 \norm{K_w}^2=\prod_{\ell=1}^d\frac{1}{1-\abs{w_\ell}^2}.
\end{equation}
For $\varphi\in H^\infty(\D^d)$, the reproducing property gives
$M_\varphi^*K_w=\overline{\varphi(w)}K_w$.

Kernels at distinct points are linearly independent. For distinct
$w_1,\ldots,w_s$, products of suitable linear coordinate factors
give polynomials $p_j$ with $p_j(w_k)=\delta_{jk}$. Pairing a
relation among the kernels with these polynomials makes every
coefficient vanish. Thus infinitely many points yield infinitely
many independent kernels, even inside a small neighborhood.

The relevant operator observation applies to several commuting
operators at once. For $T=(T_1,\ldots,T_q)$, write
$T^\alpha=T_1^{\alpha_1}\cdots T_q^{\alpha_q}$.

\begin{lemma}\label{lem:joint-eigenspace}
Let $T_1,\ldots,T_q$ be commuting bounded operators on $\Hh$, and
let $f_1,\ldots,f_r\in\Hh$. For $\lambda\in\C^q$, put
\[
 E_\lambda=\bigcap_{\ell=1}^q
          \ker(T_\ell^*-\overline{\lambda_\ell}I).
\]
Then
\[
 E_\lambda\cap\Span\{f_1,\ldots,f_r\}^{\perp}
 \subseteq
 \left(\overline{\Span}\{T^\alpha f_j:
              \alpha\in\Nzero^q,\ 1\leq j\leq r\}\right)^\perp.
\]
If $E_\lambda$ is infinite-dimensional, the orbit span has infinite
codimension.
\end{lemma}

\begin{proof}
If $h$ belongs to the intersection on the left, then
\[
 \inner{T^\alpha f_j}{h}=\inner{f_j}{(T^\alpha)^*h}=0,
\]
since the vector $(T^\alpha)^*h$ is a scalar multiple of $h$.
Imposing $r$ orthogonality conditions removes at most $r$ dimensions
from $E_\lambda$.
\end{proof}

The lemma says that many eigenvectors for one common eigenvalue
require many independent initial vectors. For multipliers, common
level sets provide precisely such eigenvectors. We state the result
for a tuple of symbols, so that the scalar case and the parameter
threshold can be treated together.

\begin{theorem}\label{thm:main}
Let $1\leq q<d$, let
$\Phi=(\varphi_1,\ldots,\varphi_q)\in H^\infty(\D^d)^q$, and let
$F=(f_1,\ldots,f_r)\in H^2(\D^d)^r$. Set
\[
 \M_\Phi(F)=\overline{\Span}\{
   \varphi_1^{\alpha_1}\cdots\varphi_q^{\alpha_q}f_j:
          \alpha\in\Nzero^q,\ 1\leq j\leq r\}.
\]
Then $\codim\M_\Phi(F)=\infty$.
\end{theorem}

\begin{proof}
We first find an infinite common level set of $\Phi$. Let
\[
 k=\max_{z\in\D^d}\rank D\Phi(z)
\]
and choose a point $a$ where this rank is attained. If $k=0$, all the symbols are constant, and
the whole polydisk is a common level set. If $k>0$, a nonzero
$k\times k$ minor remains nonzero near $a$, so the rank is constant
there. The holomorphic constant-rank theorem gives a local level set
of complex dimension $d-k\geq d-q>0$.

For clarity, the local-coordinate argument behind this last step is
as follows. After relabeling components and coordinates, the map
\[
 z\mapsto(\varphi_1(z),\ldots,\varphi_k(z),z_{k+1},\ldots,z_d)
\]
is locally invertible. In these coordinates, the remaining components
of $\Phi$ depend only on the first $k$ coordinates, since otherwise
the rank would exceed $k$. Fixing those coordinates therefore leaves
$d-k$ free complex coordinates in the common level set.

Put $\lambda=\Phi(a)$ and let $V$ be the infinite level set just
obtained. By \eqref{eq:kernel},
\[
 \Span\{K_w:w\in V\}
 \subseteq\bigcap_{\ell=1}^q
       \ker(M_{\varphi_\ell}^*-\overline{\lambda_\ell}I).
\]
The kernels are independent, so this common eigenspace is
infinite-dimensional. The result follows from
Lemma~\ref{lem:joint-eigenspace}.
\end{proof}

The level-set argument concerns values of functions and does not
require a bounded multiplication operator. To make this explicit,
let $\Phi=(\varphi_1,\ldots,\varphi_q)$ be a tuple of holomorphic
functions and put
\[
 \begin{gathered}
 \mathcal D_\infty(\Phi)
 =\{f\in H^2(\D^d):\Phi^\alpha f\in H^2(\D^d)
                         \text{ for every }\alpha\in\Nzero^q\},\\
 \Phi^\alpha=\prod_{\ell=1}^q\varphi_\ell^{\alpha_\ell}.
 \end{gathered}
\]
This is the common domain on which all joint iterates are defined.

\begin{corollary}\label{cor:unbounded}
Let $1\leq q<d$, let the components of $\Phi$ be holomorphic on
$\D^d$, and let $f_1,\ldots,f_r\in\mathcal D_\infty(\Phi)$.
Then $\codim\M_\Phi(F)=\infty$, where the closure defining
$\M_\Phi(F)$ is taken in the $H^2$ norm.
\end{corollary}

\begin{proof}
The common-level-set construction in Theorem~\ref{thm:main} uses
only holomorphy. Choose distinct $w_1,\ldots,w_N$ in an infinite
fiber $\Phi^{-1}(\lambda)$. Polynomial interpolation shows that
the bounded evaluation map
\[
 R_Nf=(f(w_1),\ldots,f(w_N))
\]
maps $H^2$ onto $\C^N$.
Since $R_N(\Phi^\alpha f_j)=\lambda^\alpha R_Nf_j$, continuity
gives
\[
 R_N\M_\Phi(F)\subseteq\Span\{R_Nf_1,\ldots,R_Nf_r\}.
\]
The induced map from $H^2/\M_\Phi(F)$ onto
$\C^N/\Span\{R_Nf_j:1\leq j\leq r\}$ yields codimension at
least $N-r$. The integer $N$ is arbitrary.
\end{proof}

For an unbounded symbol, multiplication has the natural domain
$\mathcal D(M_\varphi)=\{f\in H^2:\varphi f\in H^2\}$.
The corollary makes no density assumption on this domain and does
not use an unbounded adjoint. Jury and Martin \cite{JuryMartin}
study densely defined multiplication operators and their domains
in complete Pick spaces; the argument above uses independent
point evaluations on an infinite fiber. As a simple unbounded
example, take $\varphi(z)=-\log(1-z_1)$. Every power of this
function belongs to $H^2$, since its boundary singularity is
logarithmic. Hence every polynomial lies in
$\mathcal D_\infty(\varphi)$, so this common domain is dense.

\begin{example}\label{ex:threshold}
The number of parameters in Theorem~\ref{thm:main} is sharp.
The coordinate multipliers $M_{z_1},\ldots,M_{z_d}$ have the joint
orbit $\{z^\alpha:\alpha\in\Nzero^d\}$ of $1$, which is an
orthonormal basis. If only the first $q<d$ coordinates are used,
the orbit of $1$ consists of functions independent of the remaining
coordinates. Theorem~\ref{thm:main} shows that replacing the
coordinates by arbitrary bounded holomorphic symbols, and allowing
finitely many initial functions, cannot remove this deficit.
\end{example}

For a single multiplier, the theorem gives \eqref{eq:scalar-result}
and rules out similarity to any finitely multicyclic operator:
an invertible intertwiner transports complete orbit families.
This includes finite direct sums of the block-diagonal Carleson
operators of Krishtal and Miller
\cite[Theorem~1.2]{KrishtalMillerBDC}, which have single generating
orbits. For nonconstant $\varphi$, this particular exclusion also
follows from $(\varphi-\lambda)f=0\Rightarrow f=0$, whereas
Jordan blocks have eigenvectors. That simpler observation applies
even in one variable.

\begin{remark}\label{rem:berger-shaw}
The bounded scalar assertion \eqref{eq:scalar-result} also follows
from the Berger--Shaw theorem \cite{BergerShaw}: a finitely
multicyclic hyponormal operator has a trace-class self-commutator.
The multiplier $M_\varphi$ is subnormal, being the restriction of
normal multiplication on $L^2(\mathbb T^d)$ to the invariant
Hardy space, and hence is hyponormal. Write
$\varphi(z)=\sum_\alpha a_\alpha z^\alpha$. On the monomial
orthonormal basis,
\[
 \inner{(M_\varphi^*M_\varphi-M_\varphi M_\varphi^*)z^\beta}
       {z^\beta}
 =\norm{\varphi}^2-\sum_{\alpha\leq\beta}\abs{a_\alpha}^2
 =\sum_{\alpha\not\leq\beta}\abs{a_\alpha}^2,
\]
where $\alpha\leq\beta$ means coordinatewise inequality.
If $\varphi$ is nonconstant, choose $a_\gamma\ne0$ with
$\gamma_i>0$, and choose $j\ne i$. The diagonal entries on
$1,z_j,z_j^2,\ldots$ are all at least $\abs{a_\gamma}^2>0$.
Thus the self-commutator is not compact, which excludes finite
multicyclicity. If a finite family of orbits had a closed span of finite
codimension, adding a basis of its orthogonal complement to the
initial vectors would give a finite multicyclic family. Constant
symbols have finite-dimensional orbit spans directly.
\end{remark}

\section{The Jury product and finite Taylor sections}\label{sec:sections}

We now follow the coefficients rather than the values of the
functions. The advantage is that the obstruction becomes a count of
vectors in a finite-dimensional space. We begin with the bidisk,
where the coefficients can be arranged in a matrix.

For $A=(a_{ij}),B=(b_{ij})\in\C^{M\times N}$, indexed from zero,
define their truncated convolution by
\begin{equation}\label{eq:jury}
 (A\jury B)_{ij}
 =\sum_{p=0}^i\sum_{s=0}^j a_{ps}b_{i-p,j-s},
 \qquad 0\leq i<M,\quad 0\leq j<N.
\end{equation}
This is the \emph{Jury product}. The name recalls Jury's treatment
of Carath\'eodory--Fej\'er and Nevanlinna--Pick interpolation through
matrix products and positive-matrix extension problems
\cite{JuryThesis}. In particular, this convolution preserves positive
semidefiniteness for square matrices. Mashreghi, Nasri, and
Vishwakarma \cite{MashreghiNasriVishwakarma} developed its functional
calculus, positivity properties, and Cayley--Hamilton theory.
Related inequalities for causal and generalized operator products
were obtained by Guillot, Mashreghi, and Vishwakarma
\cite{GuillotMashreghiVishwakarmaOppenheim,GuillotMashreghiVishwakarmaSharp};
a broader completely positive framework is studied by Evert, Jury,
and McCullough \cite{EvertJuryMcCullough}. For our orbit problem,
the useful feature is the algebraic structure of convolution.

Associate a coefficient matrix with its polynomial
$\sum_{i<M,j<N}a_{ij}x^iy^j$. Formula~\eqref{eq:jury} then identifies
matrix convolution with multiplication in
\begin{equation}\label{eq:algebra}
 \A_{M,N}=\C[x,y]/(x^M,y^N).
\end{equation}
Terms crossing the boundary of the rectangle are discarded.
Thus the product is associative and commutative, with identity
corresponding to the matrix whose only nonzero entry is $a_{00}=1$.
We use this polynomial identification below and write products in
$\A_{M,N}$ by juxtaposition.

Write $T=c1+U$, where $c$ is the constant coefficient and $U$ has
zero constant coefficient. Every monomial of $U$ has positive total
degree. Any monomial of degree at least $M+N-1$ is divisible by
$x^M$ or $y^N$, so
\begin{equation}\label{eq:nilpotence}
 U^{M+N-1}=0.
\end{equation}
This is the Cayley--Hamilton identity of
\cite[Theorem~F]{MashreghiNasriVishwakarma}, expressed in the
polynomial model. Denote the smallest positive integer $\nu$ with
$U^\nu=0$ by $\nu(U)$; in particular, $\nu(0)=1$.
Since $T=c1+U$ and $U=T-c1$, polynomials in $T$ and in $U$ have
the same span. The orbit of any initial matrix $B$ is therefore
spanned by $B,UB,\ldots,U^{\nu(U)-1}B$.

Before passing back to the Hardy space, let us record the standard
module-theoretic criterion for the number of initial matrices, or
\emph{probes}, needed by the finite model.
Modulo $U\A_{M,N}$, the evolution $T$ acts as the scalar $c$.
An orbit contributes only one direction to this quotient, and this
simple necessary condition is also sufficient.

\begin{proposition}\label{prop:probes}
Let $T=c1+U\in\A_{M,N}$ with $U$ of zero constant coefficient.
The orbits of $B_1,\ldots,B_r$ span $\A_{M,N}$ if and only if their
residue classes span $\A_{M,N}/U\A_{M,N}$. Hence the minimum
number of probes is
\[
 \mu(U)=\dim_\C(\A_{M,N}/U\A_{M,N}),
 \qquad
 \mu(U)\geq\left\lceil\frac{MN}{\nu(U)}\right\rceil.
\]
\end{proposition}

\begin{proof}
Let $\mathcal S=\Span\{U^kB_j:k\geq0,\ 1\leq j\leq r\}$,
which is also the span of the $T$-orbits. Necessity follows on
passing to the quotient. Conversely, if the residue classes span,
then $\A_{M,N}=\mathcal S+U\A_{M,N}$. The invariance
$U\mathcal S\subseteq\mathcal S$ allows us to iterate this identity:
\[
 \A_{M,N}=\mathcal S+U^k\A_{M,N}\qquad(k\geq1).
\]
For $k=\nu(U)$, the remainder vanishes. Thus representatives of a
quotient basis give exactly $\mu(U)$ probes. Each orbit has dimension
at most $\nu(U)$, so $MN\leq\mu(U)\nu(U)$.
\end{proof}

This is the nilpotent form of Nakayama's lemma \cite[Proposition~2.6]{AtiyahMacdonald},
applied to
$\A_{M,N}$ as a module over $\C[t]/(t^{\nu(U)})$, with $t$
acting by multiplication by $U$. Equivalently, $\mu(U)$ counts the
Jordan blocks of the
linear operator $A\mapsto UA$. For example, $U=x$ gives
$\A_{M,N}/U\A_{M,N}\cong\C[y]/(y^N)$, so exactly $N$ probes
are needed. For $U=x+y$, substitution of $y=-x$ gives
\[
 \A_{M,N}/U\A_{M,N}
 \cong\C[x]/(x^M,x^N),
 \qquad \mu(x+y)=\min(M,N).
\]
These formulas describe a finite reconstruction problem. To prove
infinite codimension in the Hardy space, it is enough to count the
vectors contributed by a fixed number of probes.

For a function holomorphic on $\D^2$, define its rectangular Taylor
truncation by
\[
 \pi_{M,N}\left(\sum_{i,j\geq0}a_{ij}z_1^iz_2^j\right)
       =\sum_{i<M,j<N}a_{ij}x^iy^j\in\A_{M,N}.
\]
Truncation respects products of holomorphic functions. On $H^2$,
it is the orthogonal projection onto the corresponding polynomial
space, when $\A_{M,N}$ has the coefficient norm. These two
properties allow the finite-dimensional calculation to detect
vectors orthogonal to the entire orbit.

\begin{theorem}\label{thm:finite-section}
Let $\varphi$ be holomorphic on $\D^2$, let
$F=(f_1,\ldots,f_r)\in\mathcal D_\infty(\varphi)^r$, and set
$\M=\overline{\Span}\{\varphi^nf_j:n\geq0,\ 1\leq j\leq r\}$.
Put $\M_{M,N}=\pi_{M,N}(\M)$ and
$U=\pi_{M,N}\varphi-\varphi(0)1$. Then
\begin{equation}\label{eq:section-bound}
 \codim_{\A_{M,N}}\M_{M,N}
 \geq MN-r\nu(U)\geq MN-r(M+N-1).
\end{equation}
In particular, $\codim\M=\infty$.
\end{theorem}

\begin{proof}
Write $T=\pi_{M,N}\varphi$ and $B_j=\pi_{M,N}f_j$.
Multiplicativity gives $\pi_{M,N}(\varphi^nf_j)=T^nB_j$.
The span of these truncated vectors is finite-dimensional and hence
closed. Continuity of $\pi_{M,N}$ therefore gives
\[
 \begin{aligned}
 \M_{M,N}
 &=\Span\{T^nB_j:n\geq0,\ 1\leq j\leq r\}\\
 &=\Span\{U^kB_j:0\leq k<\nu(U),\ 1\leq j\leq r\}.
 \end{aligned}
\]
There are at most $r\nu(U)$ vectors in the last family, which proves
\eqref{eq:section-bound}.

Identify $\A_{M,N}$ with
$\Hh_{M,N}=\Span\{z_1^iz_2^j:i<M,\ j<N\}\subset H^2(\D^2)$.
For $h\in\Hh_{M,N}\ominus\M_{M,N}$ and $f\in\M$,
$\inner{f}{h}=\inner{\pi_{M,N}f}{h}=0$.
Thus $\Hh_{M,N}\ominus\M_{M,N}\subseteq\M^\perp$ and
\[
 \dim\M^\perp\geq n^2-r(2n-1)
 \qquad(M=N=n).
\]
The right side tends to infinity.
\end{proof}

\begin{example}\label{ex:sharp-bound}
Take $\varphi(z_1,z_2)=z_1+z_2$ and $f_1=1$.
In $\A_{M,N}$ the powers $(x+y)^k$, $0\leq k\leq M+N-2$,
are nonzero: for every such $k$, some term $x^iy^{k-i}$ survives
the truncation. They are independent because their total degrees
are different. Consequently,
\[
 \dim\M_{M,N}=M+N-1,
 \qquad \codim\M_{M,N}=(M-1)(N-1).
\]
The estimate in \eqref{eq:section-bound} is therefore sharp for a
single orbit. On square sections its relative codimension tends
to one, although the full orbit contains infinitely many independent
functions.
\end{example}

The same counting argument applies to joint orbits, including the
unbounded symbols in Corollary~\ref{cor:unbounded}. For positive
integers $N_1,\ldots,N_d$, put
\[
 \A_{\mathbf N}=\C[x_1,\ldots,x_d]/(x_1^{N_1},\ldots,x_d^{N_d}),
 \qquad L_{\mathbf N}=1+\sum_{\ell=1}^d(N_\ell-1),
\]
and let $\pi_{\mathbf N}$ denote the corresponding Taylor truncation.
Every product of $L_{\mathbf N}$ elements with zero constant
coefficient vanishes in this algebra.

\begin{proposition}\label{prop:joint-sections}
For $\Phi,F$ as in Corollary~\ref{cor:unbounded}, the space
$\M_{\mathbf N}=\pi_{\mathbf N}(\M_\Phi(F))$ satisfies
\begin{equation}\label{eq:joint-bound}
 \codim_{\A_{\mathbf N}}\M_{\mathbf N}
 \geq \prod_{\ell=1}^dN_\ell
       -r\binom{L_{\mathbf N}+q-1}{q}.
\end{equation}
\end{proposition}

\begin{proof}
Set $U_\ell=\pi_{\mathbf N}\varphi_\ell-\varphi_\ell(0)1$ for
$1\leq\ell\leq q$, and $B_j=\pi_{\mathbf N}f_j$.
Subtracting the constant coefficients preserves the joint polynomial
orbit span. As in Theorem~\ref{thm:finite-section}, this identifies
$\M_{\mathbf N}$ with the span of $U^\alpha B_j$.
Only $\abs{\alpha}<L_{\mathbf N}$ can contribute. There are
$\binom{L_{\mathbf N}+q-1}{q}$ such multi-indices, which proves
\eqref{eq:joint-bound}.
\end{proof}

For $N_1=\cdots=N_d=n$, the ambient dimension is $n^d$, whereas
the upper bound for the orbit dimension is
$r\binom{d(n-1)+q}{q}=O(n^q)$. If $q<d$, the codimension tends
to infinity and the relative codimension tends to one. The
orthogonal-projection argument in
Theorem~\ref{thm:finite-section} now gives a second proof of
Theorem~\ref{thm:main} and Corollary~\ref{cor:unbounded}.
In this form the two arguments express the
same dimension deficit: positive-dimensional fibers in the first
proof, and too few surviving monomials in the second.

\section{Interpolation and Hardy-space models}\label{sec:kernels}

Interpolation provides a useful comparison with the preceding
obstruction. Given distinct points $Z=(a_j)\subset\D^d$, write
$\kappa_{a_j}=K_{a_j}/\norm{K_{a_j}}$.
The normalized kernels form a \emph{Riesz sequence}
\cite[Chapter~3]{Christensen} if there are
$0<A\leq B<\infty$ such that
\begin{equation}\label{eq:riesz}
 A\sum_j\abs{c_j}^2
 \leq\norm{\sum_jc_j\kappa_{a_j}}^2
 \leq B\sum_j\abs{c_j}^2
\end{equation}
for every finitely supported scalar sequence $c$.
A Riesz sequence is a Riesz basis for its own closed span; it need
not span the ambient Hardy space. In fact, point kernels cannot
form a Riesz basis for the full Hardy space in any dimension.

The following elementary fact explains why. A sequence is
\emph{minimal} if no member belongs to the closed span of the others;
every Riesz sequence is minimal.

\begin{proposition}\label{prop:minimal-kernels}
Let $d\geq1$. If $(K_{a_j})$ is minimal in $H^2(\D^d)$, then its
closed span has infinite codimension. The same conclusion holds
whenever $Z$ is $H^\infty(\D^d)$-interpolating.
\end{proposition}

\begin{proof}
Minimality gives $F\in H^2(\D^d)$ with $F(a_1)\ne0$ and
$F(a_j)=0$ for $j\ne1$, by orthogonal separation of $K_{a_1}$
from the closed span of the remaining kernels. Define
\[
 h(z)=(z_1-a_{1,1})F(z).
\]
This is a nonzero function in $H^2$ that vanishes on all of $Z$.
Each $z_1^nh$, $n\geq0$, also belongs to $H^2$ and vanishes on $Z$,
so it is orthogonal to every $K_{a_j}$. These functions are linearly
independent: a relation would make a polynomial in $z_1$ vanish
on the open set where $h\ne0$, and hence identically.

If $Z$ is $H^\infty$-interpolating, bounded data on $Z$ have bounded
holomorphic interpolants. In particular, one can choose $F$ above
with $F(a_1)=1$ and $F(a_j)=0$ for $j\ne1$, and repeat the argument.
\end{proof}

Thus kernel incompleteness is already present in one variable and
has its own elementary explanation. Interpolation nevertheless
produces useful frames on the associated sequence space. We give
the details to make the ambient spaces explicit.

First, $H^\infty$-interpolation implies that the normalized kernels
are a Riesz sequence. Indeed, the open mapping theorem gives a
constant $C$ such that every sequence of signs $\varepsilon_j=\pm1$
has an interpolant $\psi_\varepsilon$ with
$\norm{\psi_\varepsilon}_\infty\leq C$. For finite sums
$x=\sum_jc_j\kappa_{a_j}$, multiplication by
$M_{\psi_\varepsilon}^*$ changes $x$ into
$x_\varepsilon=\sum_j\varepsilon_jc_j\kappa_{a_j}$ and changes
$x_\varepsilon$ back into $x$. Hence
$C^{-1}\norm{x}\leq\norm{x_\varepsilon}\leq C\norm{x}$.
Averaging the squared norms over independent signs gives
\eqref{eq:riesz} with $A=C^{-2}$ and $B=C^2$.
For the bidisk, the more comprehensive interpolation theorem of
Agler and McCarthy \cite{AglerMcCarthy} uses uniform Gram-matrix
estimates for all admissible kernels. Here we need only the ordinary
Szeg\H{o} kernel.

Put
\[
 w_j=\norm{K_{a_j}}^{-1}
     =\prod_{\ell=1}^d(1-\abs{a_{j,\ell}}^2)^{1/2},
 \qquad g=(w_j)_j,
 \qquad D_\ell e_j=a_{j,\ell}e_j
\]
on $\ell^2$, where $(e_j)$ is its standard basis. The following
equivalence is the multivariable version of the usual
kernel-to-orbit interpretation of Carleson frames; compare
\cite{AldroubiCabrelliCakmakMolterPetrosyan,KrishtalMillerBDC}.

\begin{proposition}\label{prop:multi-orbit}
The normalized kernels $(\kappa_{a_j})$ form a Riesz sequence in
$H^2(\D^d)$ if and only if $g\in\ell^2$ and
$\{D_1^{\alpha_1}\cdots D_d^{\alpha_d}g:
\alpha\in\Nzero^d\}$ is a frame for $\ell^2$. The same bounds
work in the Riesz inequalities and the frame inequalities.
\end{proposition}

\begin{proof}
Expansion of the kernels in the monomial basis gives, for every
finitely supported $c$,
\begin{equation}\label{eq:coefficient-identity}
 \norm{\sum_jc_j\kappa_{a_j}}^2
 =\sum_{\alpha\in\Nzero^d}
       \left|\sum_jc_jw_j\overline{a_j^\alpha}\right|^2.
\end{equation}
If the kernels have an upper Riesz bound $B$, the term $\alpha=0$
shows that $|\sum_jc_jw_j|^2\leq B\norm{c}_{\ell^2}^2$.
Therefore $g\in\ell^2$, and every $D^\alpha g$ belongs to $\ell^2$.
The right side of \eqref{eq:coefficient-identity} is now
$\sum_\alpha|\inner{c}{D^\alpha g}|^2$. Its upper bound defines
a bounded analysis operator on the dense subspace of finitely
supported sequences, and thus on all of $\ell^2$. The lower bound
extends by continuity as well. This proves the frame inequalities.
Conversely, if $g\in\ell^2$ and the orbit is a frame, the same
identity immediately gives both Riesz bounds.
\end{proof}

Equivalently, under these conditions the normalized evaluation map
$E_Zf=(w_jf(a_j))_j$ is bounded and onto $\ell^2$.
It sends the monomial basis to the joint diagonal orbit and has
$E_Z^*e_j=\kappa_{a_j}$. Proposition~\ref{prop:minimal-kernels}
shows that its kernel is infinite-dimensional. Thus the frame
describes the interpolation data, while the Hardy space contains
many additional functions invisible to those data.

\begin{example}\label{ex:embedded-carleson}
Let $(\lambda_j)\subset\D$ be a one-variable
$H^\infty$-interpolating sequence, also called a Carleson sequence,
and put $a_j=(\lambda_j,0,\ldots,0)$. Interpolants depending only
on $z_1$ show that $Z=(a_j)$ is $H^\infty(\D^d)$-interpolating.
Here $g_j=\sqrt{1-\abs{\lambda_j}^2}$ and $D_\ell=0$ for
$\ell\geq2$, so Proposition~\ref{prop:multi-orbit} says that
$\{D_1^ng:n\geq0\}$ alone is a frame for $\ell^2$.
This is compatible with Theorem~\ref{thm:main}: the sharp lower
bound of $d$ multiplier parameters concerns generation of the
\emph{full} space $H^2(\D^d)$ by finitely many initial functions.
\end{example}

The role of the ambient space is also visible in the model-space
representation of dynamical frames. Christensen, Hasannasab, and
Philipp \cite[Theorem~3.4]{ChristensenHasannasabPhilipp} obtained a scalar
compressed-shift representation, and Cabrelli, Molter, and Su\'arez
\cite[Theorem~2.1]{CabrelliMolterSuarez} established its vector-valued
form. We recall the construction for finitely many generators.
If \eqref{eq:orbit} is a frame, its synthesis operator
\[
 Q:H^2(\D;\C^r)\longrightarrow\Hh,
 \qquad Q(z^ne_j)=T^nf_j,
\]
is bounded and onto. Here $e_j$ are the constant coordinate vectors,
and $S$ denotes multiplication by $z$. The identity $QS=TQ$ makes
$\ker Q$ invariant under $S$. Thus
$\mathcal K=(\ker Q)^\perp$ is a \emph{model space}, meaning a
closed subspace invariant under $S^*$, and $V=Q|_{\mathcal K}$
is boundedly invertible. With $P_{\mathcal K}$ the orthogonal
projection and $S_{\mathcal K}=P_{\mathcal K}S|_{\mathcal K}$,
\begin{equation}\label{eq:frame-model}
 T=VS_{\mathcal K}V^{-1},\qquad
 f_j=VP_{\mathcal K}e_j,\qquad
 S_{\mathcal K}^nP_{\mathcal K}e_j=P_{\mathcal K}(z^ne_j).
\end{equation}
The projected monomials form a Parseval frame for $\mathcal K$,
that is, a frame with both bounds equal to one.
Conversely, an invertible image of this frame gives a frame of
iterations. Aguilera, Cabrelli, Negreira, and Paternostro
\cite{AguileraCabrelliNegreiraPaternostro} use such models to
determine the minimum number of frame generators.

Scalar examples are
$K_\theta=H^2(\D)\ominus\theta H^2(\D)$ for an inner function
$\theta$, that is, a bounded analytic function with unimodular
boundary values almost everywhere. The compression to $K_\theta$ has the Parseval orbit
$\{P_{K_\theta}z^n:n\geq0\}$ of $P_{K_\theta}1$.
The full one-variable Hardy space is another model space.
These frames live on compressed spaces. By Theorem~\ref{thm:main},
a bounded scalar multiplier on the full $H^2(\D^d)$, $d\geq2$,
cannot be similar to any such model with finite coefficient multiplicity.

Krishtal, Mashreghi, and Miller \cite{KrishtalMashreghiMillerSectorial}
use the representation 
\eqref{eq:frame-model} for invertible
operators with spectrum in a sector about the positive real axis,
of opening less than $2\pi$, inside the unit disk.
The principal logarithm defines $T^s$.
For nonnegative sampling times with uniformly bounded counts in
unit intervals, they relate synthesis to projected vector-valued
exponentials on a suitable arc, up to an invertible transformation
and a compact perturbation. That setting begins with a frame of integer iterates on
its model space; the present obstruction identifies full polydisk
multiplication operators for which no such finite family exists.

\section*{Acknowledgements}

Ilya Krishtal was supported in part by the Fulbright Global Scholar Award.
Javad Mashreghi was supported by the Canada Research Chairs Program
(CRC-2022-00097) and the Natural Sciences and Engineering Research Council
of Canada (Discovery Grant RGPIN-2024-04232).

ChatGPT (OpenAI) was used to assist with drafting and revising the
exposition, searching the literature, and checking selected arguments.
The authors are responsible for all parts of the manuscript.

\end{document}